\documentclass[10pt]{amsart}

\usepackage{cite}
\usepackage{graphicx,epsfig}
\usepackage{amsmath,amsfonts,amssymb,amsthm,amsbsy}
\usepackage{color}

\def\R{\mathbb{R}}

\def\xta{\langle \xi,\eta \rangle}

\theoremstyle{plain}
\newtheorem{thm}{Theorem}
\newtheorem{lemma}{Lemma}
\newtheorem{cor}{Corollary}
\newtheorem{rem}{Remark}

\theoremstyle{definition}

\begin{document}

\title[New monotonicity formulae]{New monotonicity formulas for the curve shortening flow in $\R^3$}

\author{Hayk Mikayelyan}

\date{\today}

\address{\noindent Mathematical Sciences\\
The University of Nottingham Ningbo China\\
Taikang Dong Lu Nr. 199,
Ningbo 315100\\ PR China}

\email{\noindent Hayk.Mikayelyan@nottingham.edu.cn}

\maketitle

\begin{abstract}
We apply the stabilization technique, developed by T. Zelenyak in 1960s for parabolic equations, on 
curve shortening flow in $\R^3$, and derive several new monotonicity formulas. All 
of them share one main feature: the dependence of the ``energy'' term on 
the angle between the position vector and the plane orthogonal
to the tangent vector. The first formula deals with the projection of the curve on the unit sphere,
and computes the derivative of its length.
The second formula 
is the generalization of the
classical formula of G. Huisken, while the third one  
is the generalization of the monotonicity formula with logarithmic terms previously
derived by the author for plane curves, \cite{M}.
\end{abstract}



\section{ Introduction}
The curve shortening problem is one of the most beautiful and classical problems in geometric PDEs. It models a curve
moving by its curvature vector, and questions like singularity formation, long time asymptotics of rescaled solutions, existence of ancient solutions, etc., have 
been in the focus of research in past decades.  
We give a short introduction of the problem in Section \ref{sec:probset}, and would like to 
refer the reader to the following book and lecture notes \cite{CZ, Ha}, as well as some 
important results \cite{AL, AAAW, A, Ang, AngV, DHS, E, Ga, GH, G, Hu, Hu1}, for a comprehensive introduction to the topic.

In this article we develop ideas from \cite{M}, 
where a new monotonicity formula has been derived for the curve shortening flow in the plane, 
to obtain several  
monotonicity formulas in $\R^3$. The idea is based on techniques from the article by T. Zelenyak \cite{Z}, where 
general monotonicity formulas for parabolic problems on an interval have been derived. 
This allows to compute the time derivative of certain ``energies'' depending
on the angle between the position vector and the plane orthogonal
to the tangent vector. 
In $\R^2$ the angle between the position vector and the normal vector has been considered 
mainly in the context of the support functions of convex curves (see \cite{GH}, \cite{CZ}). 
To the best knowledge of the author, however, 
monotonicity formulas involving the support functions have not been considered. 

Our first monotonicity formula, 
see Section~\ref{sec:mainres},
deals with the projection of the curve on the unit sphere,
and computes the derivative of its length. Since this quantity 
is constant for plane, star-shaped curves (see Corollary \ref{thm3} on page \pageref{thm3}), the right hand side of our formula 
measures how much a curve in $\R^3$ deviates from being plane and ``star-shaped''. 

Huisken's monotonicity formula  (see \cite{Hu}) plays a crucial role in the theory 
of flows driven by the mean curvature, in particular the curve shortening flow 
(see the equation (\ref{monformHu}) on page \pageref{monformHu}). 
The second formula we derive 
is the generalization of this
classical formula, which is not covered by the larger class of monotonicity formulas 
introduced in \cite{Hu1}, Corollary 4.2.

Our third formula  
is the generalization of the monotonicity formula with logarithmic terms, previously
derived by the author for the star-shaped plane curves, see \cite{M}.

The results we present provide more involved tools for the analysis of the stability of the curves, 
and we believe that the method introduced can be applied for other geometric co-dimension two problems in $\R^3$.


\subsection{Problem setting}\label{sec:probset}
We consider a closed curve in $\mathbb{R}^3$ moving by its curvature
$$
\partial_t \gamma = \kappa\nu,
$$
where $\gamma: (0,T) \times S^1\to \mathbb{R}^3$ is the curve parametrization,
\begin{equation}\label{kappa}
\kappa=
\frac{|\gamma''\times\gamma'|}{|\gamma'|^3}
\end{equation}
is the curvature and
$$\nu=
\frac{\gamma'\times (\gamma''\times\gamma')}{|\gamma'||\gamma''\times\gamma'|}
$$
is the normal vector. Here $ ' $ means the derivative in $x\in S^1$ variable.

Assume the first singularity appears at point $0$ after finite time $T$.
We rescale the parametrization in the following way
$$
\tau=-\log(T-t), \,\,\,\,\,\,\, \tilde{\gamma}(\tau,x)=(T-t)^{-\frac{1}{2}}\gamma(t,x)
$$
and arrive at
\begin{equation}  \label{main}
\partial_\tau\tilde{\gamma}=\frac{1}{2}{\tilde \gamma}+
\frac{{\tilde \gamma}'\times ({\tilde \gamma}''\times{\tilde \gamma}')}{|{\tilde \gamma}'|^4},
\end{equation}
which is going to be the main equation we consider in this article.

Throughout the paper 
$$
\psi=\arcsin \frac{\langle {\tilde\gamma},{\tilde\gamma}' \rangle}{|{\tilde\gamma}||{\tilde\gamma}'|}\in[-\tfrac{\pi}{2},\tfrac{\pi}{2}]
$$
will denote the angle between the position vector $ {\tilde\gamma}$ and the plane orthogonal
to the tangent vector $ {\tilde\gamma}'$, 
with a sign coming from the sign of $\langle {\tilde\gamma},{\tilde\gamma}' \rangle$. 
\vspace{3mm}

The paper is organized as follows: in Section~\ref{sec:mainres} the main results are introduced,
in Section~\ref{stabtech} the stabilization technique is presented, and in Section~\ref{sec:Hu}
this technique is illustrated on the classical formula of G. Huisken.
In Section~\ref{sec:computD}
some ``heavy'' computations of the so-called remainder terms are conducted, and in Section~\ref{sec:newmon}
the proofs of the results are derived from these computations.

\vspace{3mm}

{\bf Acknowledgments.} The author would like to express his gratitude and appreciation to Sigurd Angenent  and Gerhard Huisken for valuable feedback.


\section{Main results}\label{sec:mainres}

Let us consider the length of the projection of the curve on the unit sphere given by 
$$
\int_{S^1}
\frac{|{\tilde\gamma}'|}{|{\tilde\gamma}|}\cos\psi
dx.
$$
The next theorem establishes a monotonicity relation for this length.
\begin{thm}\label{thm2}
Let ${\tilde\gamma}$ be the rescaled curve shortening flow in (\ref{main}).
Then 
\begin{equation}\label{cosformul}
\frac{d}{d\tau}\int_{S^1}
\frac{|{\tilde\gamma}'|}{|{\tilde\gamma}|}\cos\psi
dx=
-
\int_{S^1}
\frac{|{\tilde\gamma}'|}{|{\tilde\gamma}|^3\cos^3\psi}\kappa^2
\big|\textup{Proj}_{\nu\times{\tilde\gamma}'} {\tilde\gamma}\big|^2 
dx
-
2\sum_{\psi(x)=\pm\frac{\pi}{2}}\frac{\kappa(x)}{|\tilde  \gamma(x)|},
\end{equation}
where the second sum is taken over the points where $\psi=\pm \frac{\pi}{2}$ 
(or ${\tilde\gamma}\parallel{\tilde \gamma}'$).\\
Further, let $a\in C([0,\infty))$ be an arbitrary continuous function on $[0,\infty)$. Then 
\begin{equation}\label{cosformul2}
\int_{S^1}
a(|{\tilde\gamma}|)|{\tilde\gamma}'|\sin\psi
dx=0.
\end{equation}
 \end{thm}

\begin{cor}\label{thm3} 
In the case of the plane curves the monotonicity formula (\ref{cosformul}) is trivial for the curves satisfying 
the condition
\begin{equation*}
\psi=\arcsin \frac{\langle {\tilde\gamma},{\tilde\gamma}' \rangle}{|{\tilde\gamma}||{\tilde\gamma}'|}\not=\pm\frac{\pi}{2},
\end{equation*}
and
$$
\int_{S^1}
\frac{|{\tilde\gamma}'|}{|{\tilde\gamma}|}\cos\psi
dx\equiv
2\pi m,
$$
where $m$ is the winding number. For general plane curves we have
\begin{equation}\label{cos2D}
\frac{d}{d\tau}\int_{S^1}
\frac{|{\tilde\gamma}'|}{|{\tilde\gamma}|}\cos\psi
dx=-
2\sum_{\psi(x)=\pm\frac{\pi}{2}}\frac{\kappa(x)}{|\tilde  \gamma(x)|},
\end{equation}
where the curvature $\kappa$ is defined in $\R^3$ by (\ref{kappa}), 
thus is non-negative.
\end{cor}
\begin{rem}
If we let the function $a(r)$ in (\ref{cosformul2}) converge to the Dirac function in $r_0$, we will obtain the simple fact that 
a closed curve, which intersects the sphere of radius $r_0$ in non-tangential fashion, exits the sphere and enters it in the same number of points.  
\end{rem}
\vspace{3mm}

Our second result is the generalization of the classical monotonicity formula of G. Huisken.
\begin{thm}\label{thm1}
Let ${\tilde\gamma}$ be the rescaled curve shortening flow in (\ref{main}), $\lambda>0$, and let
\begin{equation}
\psi=\arcsin \frac{\langle {\tilde\gamma},{\tilde\gamma}' \rangle}{|{\tilde\gamma}||{\tilde\gamma}'|}.
\end{equation}
Then 
\begin{multline}\label{monformHG}
\frac{d}{d\tau}\int_{S^1}F_\lambda({\tilde\gamma},{\tilde\gamma}')dx=
-\int_{S^1}
\left| \kappa+\frac{1}{2}\langle{\tilde\gamma},\nu\rangle\right|^2
\rho_\lambda({\tilde\gamma},{\tilde\gamma}')dx\\
-
\int_{S^1}
\left(\frac{1}{4}+(\lambda-1)\frac{\kappa^2}{|{\tilde\gamma}|^2\cos^2\psi}\right)
|\textup{Proj}_{\nu\times{\tilde\gamma}'} {\tilde\gamma}|^2 
\rho_\lambda({\tilde\gamma},{\tilde\gamma}')dx,
\end{multline}
where 
$$
F_\lambda(\xi,\eta)=a_\lambda(|\xi|)|\eta|f_\lambda(\psi),\,\,\,
\rho_\lambda(\xi,\eta)=a_\lambda(|\xi|)|\eta|g_\lambda(\psi),
$$

$$
a_\lambda(r)=e^{-\frac{r^2}{4\lambda}}r^{\frac{1-\lambda}{\lambda}},
\,\,\,
g_\lambda(\psi)=\left(\tfrac{1}{\cos\psi}\right)^{\frac{\lambda-1}{\lambda}}
$$
and 
\begin{equation}\label{flamb}
f_\lambda(\psi)=\sin\psi\int_0^\psi (\cos t)^\frac{1}{\lambda}dt+\lambda (\cos \psi)^{1+\frac{1}{\lambda}}.
\end{equation}
\end{thm}
The graphs of the functions $f_\lambda(\psi)$ for several values of $\lambda$ are displayed in Figure~1. Observe that $f_\lambda(0)=\lambda$.
\begin{figure}[h!]\label{fig:1}
\begin{center}
\includegraphics[scale=0.25]{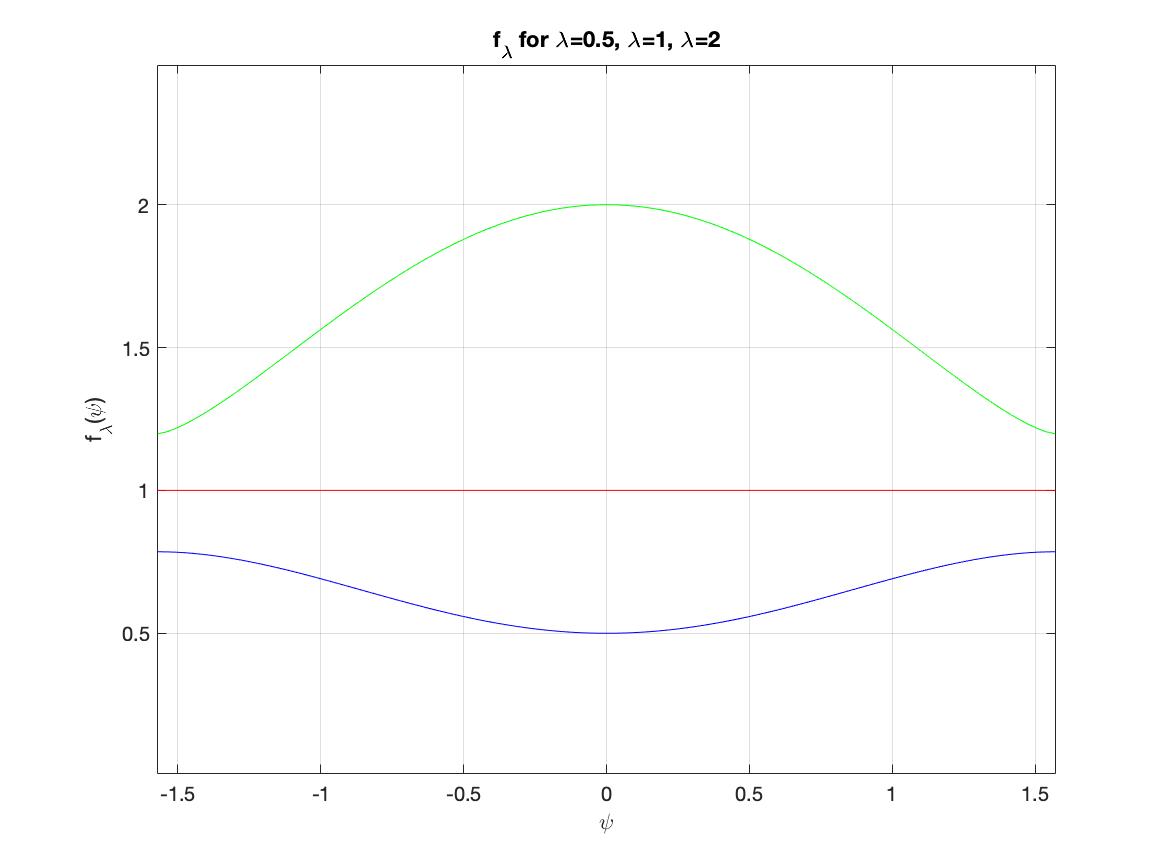}\label{fig-nc}
\caption{Functions $f_\lambda$ for $\lambda =0.5, \,\,\,1, \text{ and }\,2.$}
\end{center}
\end{figure}
\begin{rem}
For $\lambda=1$ the equation 
 (\ref{monformHG}) turns into the monotonicity formula of G. Huisken \cite{Hu} (more details in Section \ref{sec:Hu})
 \begin{equation}\label{monformHu}
\frac{d}{d\tau}\int_{S^1} |{\tilde\gamma}'|e^{-\frac{|{\tilde\gamma}|^2}{4}}dx=
-\int_{S^1}
\left(
\frac{1}{4}|\textup{Proj}_{\nu\times{\tilde\gamma}'} {\tilde\gamma}|^2
+\left| \kappa+\frac{1}{2}\langle{\tilde\gamma},\nu\rangle\right|^2
\right) |{\tilde\gamma}'|e^{-\frac{|{\tilde\gamma}|^2}{4}}
dx.
\end{equation}  
\end{rem}

\vspace{3mm}

In \cite{M}, \cite{M-cor} the version of the following formula has been derived for plane curves, 
which in $\R^2$ happens to be a monotonicity
formula. Here we generalize it in $\R^3$.

\vspace{1mm}

\begin{thm}\label{thm4}
Let ${\tilde\gamma}$ be the rescaled curve shortening flow in (\ref{main}), and let
\begin{equation}\label{noturnpsi}
\psi=\arcsin \frac{\langle {\tilde\gamma},{\tilde\gamma}' \rangle}{|{\tilde\gamma}||{\tilde\gamma}'|}\not=\pm\frac{\pi}{2}.
\end{equation}
Then 
\begin{multline}\label{oldnewmonform}
\frac{d}{d\tau}\int_{S^1}F({\tilde\gamma},{\tilde\gamma}')dx=
-\int_{S^1}
\left| \kappa+\frac{1}{2}\langle{\tilde\gamma},\nu\rangle\right|^2
\rho({\tilde\gamma},{\tilde\gamma}')
dx\\
-\int_{S^1}
\left[\frac{1}{4}-\Big( 1+|{\tilde\gamma}|b(|{\tilde\gamma}|)-\log\cos\psi\Big)
\frac{\kappa^2 }{|{\tilde\gamma}|^2\cos^2\psi} \right]
\big|\textup{Proj}_{\nu\times{\tilde\gamma}'} {\tilde\gamma}\big|^2 
\rho({\tilde\gamma},{\tilde\gamma}')
dx,
\end{multline}
where 
$$
F(\xi,\eta)=\frac{|\eta|}{|\xi|}\Big(h(\psi)-|\xi|b(|\xi|)\cos\psi\Big),
\,\,\,\rho(\xi,\eta)=\frac{|\eta|}{|\xi|}\frac{1}{\cos\psi},
$$

$$
h(\psi)=\psi\sin\psi+\cos\psi\log\cos\psi
$$ 
and 
$$
b(r)=\frac{r}{4}-\frac{\log r}{r}-\frac{1-\log 2}{2r}.
$$
\end{thm}
The graphs of the functions $f(\psi)$ and $rb(r)$ are displayed in Figure~2.
\begin{figure}[h!]\label{fig:2}
\begin{center}
\includegraphics[scale=0.34]{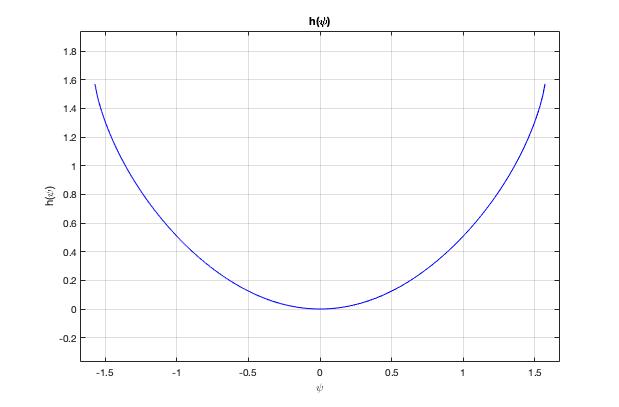}\label{fig-nc}
\includegraphics[scale=0.34]{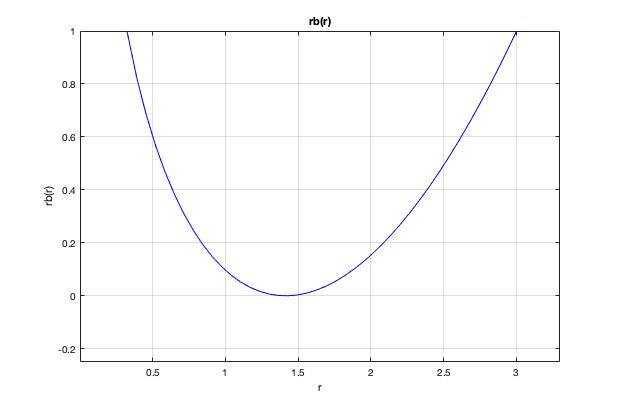}\label{fig-nc}

\caption{Functions $h(\psi)$ and $rb(r)$.}
\end{center}
\end{figure}
\begin{rem}
Observe that both $h(\psi)$ and $rb(r)$ (see Figure 2) are non-negative convex functions, and achieve their minimum value zero at 
$\psi=0$ and $r=\sqrt{2}$ respectively, which correspond to the plane circle of radius $\sqrt{2}$, i.e., 
the stable stationary plane solution of (\ref{main}). 

Moreover, for the plane circle of radius $\sqrt{2}$
in
the second term of (\ref{oldnewmonform}) 
not only $\big|\textup{Proj}_{\nu\times{\tilde\gamma}'} {\tilde\gamma}\big|$ vanishes, but also the expression
$$
\left[\frac{1}{4}-\Big( 1+|{\tilde\gamma}|b(|{\tilde\gamma}|)-\log\cos\psi\Big)
\frac{\kappa^2 }{|{\tilde\gamma}|^2\cos^2\psi} \right]
$$
does.
\end{rem}
The method of the proof in \cite{M} is interesting because it allows 
one to derive the monotonicity formula ``from nowhere'' in $\R^2$. 
The same approach would fail in $\R^3$, but one can generalize some computations from \cite{M} to $\R^3$ 
for a special class of functions, and obtain new monotonicity 
formulas. This is what we do in the next two sections.



\section{The stabilization technique}\label{stabtech}
For the system (\ref{main}) we look for functions $F(\xi, \eta)$ and $\rho(\xi, \eta)$, $\xi,\eta\in \R^3$, 
to make the following monotonicity relation possible
\begin{equation}\label{monform}
\frac{d}{d\tau}\int_{S^1} F(v_1,v_2,v_3,v_1',v_2',v_3')dx=
-\int_{S^1}|\partial_\tau \gamma|^2\rho(v_1,v_2,v_3,v_1',v_2',v_3')dx+\mathfrak{D}(\tau),
\end{equation}
where $\tilde{\gamma}(\tau,x)=\left(\begin{array}{cc} v_1(\tau,x) \\ v_2(\tau,x)\\ v_3(\tau,x)
\end{array}\right)$ and $\mathfrak{D}$ has a geometric meaning.

Differentiatng the left hand side of (\ref{monform}) and integrating by parts we
obtain under the integral
\begin{align}\label{tau2}
\partial_\tau v_1\left[\frac{\partial F}{\partial\xi_1}-
\frac{\partial^2 F}{\partial\xi_1\partial\eta_1}v_1'-
\frac{\partial^2 F}{\partial\xi_2\partial\eta_1}v_2'-
\frac{\partial^2 F}{\partial\xi_3\partial\eta_1}v_3'{\color{blue}-
\frac{\partial^2 F}{\partial\eta_1^2}v_1''-
\frac{\partial^2 F}{\partial\eta_1\partial\eta_2}v_2''-
\frac{\partial^2 F}{\partial\eta_1\partial\eta_3}v_3''}
\right]\,\,\, 
\\
+
\partial_\tau v_2\left[\frac{\partial F}{\partial\xi_2}-
\frac{\partial^2 F}{\partial\xi_1\partial\eta_2}v_1'-
\frac{\partial^2 F}{\partial\xi_2\partial\eta_2}v_2'-
\frac{\partial^2 F}{\partial\xi_3\partial\eta_2}v_3'{\color{blue}-
\frac{\partial^2 F}{\partial\eta_1\partial\eta_2}v_1''-
\frac{\partial^2 F}{\partial\eta_2^2}v_2''-
\frac{\partial^2 F}{\partial\eta_2\partial\eta_3}v_3''}
\right]\,\,\, \nonumber
\\ 
+
\partial_\tau v_3\left[\frac{\partial F}{\partial\xi_3}-
\frac{\partial^2 F}{\partial\xi_1\partial\eta_3}v_1'-
\frac{\partial^2 F}{\partial\xi_2\partial\eta_3}v_2'-
\frac{\partial^2 F}{\partial\xi_3\partial\eta_3}v_3'{\color{blue}-
\frac{\partial^2 F}{\partial\eta_1\partial\eta_3}v_1''-
\frac{\partial^2 F}{\partial\eta_2\partial\eta_3}v_2''-
\frac{\partial^2 F}{\partial\eta_3^2}v_3''}
\right]. \nonumber
\end{align}

In the first entry of the right hand side of (\ref{monform}) using (\ref{main}) we obtain under the integral
\begin{align} |\partial_\tau{\tilde \gamma}|^2=
\left(
\begin{array}{cc} \partial_\tau v_1\\\partial_\tau v_2\\\partial_\tau v_3
\end{array}
\right)
\left(  \frac{1}{2}{\tilde \gamma}+
\frac{{\tilde \gamma}'\times ({\tilde \gamma}''\times{\tilde \gamma}')}{|{\tilde \gamma}'|^4} \right)=
\\ 
\frac{1}{2}
\left(
\begin{array}{cc} \partial_\tau v_1\\\partial_\tau v_2\\\partial_\tau v_3
\end{array}
\right)\cdot
{\tilde \gamma}
+{\color{blue}
\left(
\begin{array}{cc} \partial_\tau v_1\\\partial_\tau v_2\\\partial_\tau v_3
\end{array}
\right)\cdot
\frac{{\tilde \gamma}'\times ({\tilde \gamma}''\times{\tilde \gamma}')}{|{\tilde \gamma}'|^4}}.\label{tau3}
\end{align}
Observe that
$$
\frac{{\tilde \gamma}'\times ({\tilde \gamma}''\times{\tilde \gamma}')}{|{\tilde \gamma}'|^4}=
\frac{1}{|{\tilde \gamma}'|^4}
\left(
\begin{array}{cc}
({v'}_2^2+{v'}_3^2)v_1''-v_1'v_2'v_2''-v_1'v_3'v_3''
\\
-v_1'v_2'v_1''+({v'}_3^2+{v'}_1^2)v_2''-v_2'v_3'v_3''
\\
-v_1'v_3'v_1''-v_2'v_3'v_2''+({v'}_1^2+{v'}_2^2)v_3''
\end{array}
\right)
$$
and
$$
D^2|\eta|=
|\eta|^{-3}\left(
\begin{array}{ccc}
(\eta_2^2+\eta_3^2) & -\eta_1\eta_2 & -\eta_1\eta_3
\\
-\eta_1\eta_2 & (\eta_1^2+\eta_3^2) & -\eta_2\eta_3
\\
-\eta_1\eta_3 & -\eta_2\eta_3 & (\eta_1^2+\eta_2^2)
\end{array}
\right)=|\eta|^{-3}\left( |\eta|^2 I -(\eta_i\eta_j)_{i,j}  \right).
$$
To analyse the terms containing second order derivatives in equations (\ref{tau2}) and (\ref{tau3}) we will study the action of the matrices
\begin{equation}\label{Frho}
D^2_{\eta} F(\xi, \eta)\,\,\, \text{ and }\,\,\, \rho(\xi,\eta)|\eta|^{-1}D^2|\eta|
\end{equation}
on the vector ${\tilde\gamma}''$. In the case of Huisken's monotonicity formula the function $F$ 
depends only on $|\xi|$, $|\eta|$ and two matrices coincide (see Section \ref{sec:Hu}). 
In the case of the new monotonicity formulas their difference will have rank one and 
in Section \ref{sec:d2} we will show this. 

Let us now take 
$$
\mathfrak{D}(\tau)=\int_{S^1}
\mathfrak{D}_1+
\mathfrak{D}_2dx,
$$
where $\mathfrak{D}(\tau)$ is defined by (\ref{monform}), and in $\mathfrak{D}_2$ we collect the terms containing ${\tilde\gamma}''$
\begin{equation}\label{defD2}
\mathfrak{D}_2=\partial_\tau{\tilde\gamma}
\Big[
 \rho(\xi,\eta)|\eta|^{-1}D^2|\eta|-
D^2_\eta F(\xi, \eta)
\Big]
{\tilde\gamma}'',
\end{equation}
and in  $\mathfrak{D}_1$ the remaining terms
\begin{align}\label{eqFrho}
\mathfrak{D}_1=
\partial_\tau v_1\left[\frac{\partial F}{\partial\xi_1}-
\frac{\partial^2 F}{\partial\xi_1\partial\eta_1}v_1'-
\frac{\partial^2 F}{\partial\xi_2\partial\eta_1}v_2'-
\frac{\partial^2 F}{\partial\xi_3\partial\eta_1}v_3'
\right]\\
+\partial_\tau v_2\left[\frac{\partial F}{\partial\xi_2}-
\frac{\partial^2 F}{\partial\xi_1\partial\eta_2}v_1'-
\frac{\partial^2 F}{\partial\xi_2\partial\eta_2}v_2'-
\frac{\partial^2 F}{\partial\xi_3\partial\eta_2}v_3'
\right] \nonumber
\\
+\partial_\tau v_3\left[\frac{\partial F}{\partial\xi_3}-
\frac{\partial^2 F}{\partial\xi_1\partial\eta_3}v_1'-
\frac{\partial^2 F}{\partial\xi_2\partial\eta_3}v_2'-
\frac{\partial^2 F}{\partial\xi_3\partial\eta_3}v_3'
\right] \nonumber
\\
+\rho(\xi,\eta)
\left(
\begin{array}{cc} \partial_\tau v_1\\\partial_\tau v_2\\\partial_\tau v_3
\end{array}
\right) \cdot  \frac{1}{2}{\tilde \gamma}. \nonumber
\end{align}

In Section \ref{sec:computD} we will compute $\mathfrak{D}_1$ and $\mathfrak{D}_2$
for a special class of functions $F$ and $\rho$ and will then derive several new monotonicity formulas
in Section \ref{sec:newmon}.





\section{Huisken's monotonicity formula}\label{sec:Hu}

As an intermediate step let us verify Huisken's monotonicity formula in our setting. If we take
$$
F(\xi,\eta)=\rho(\xi,\eta)=|\eta|e^{-\frac{|\xi|^2}{4}}
$$
then the matrices in (\ref{Frho}) will coincide, implying that $\mathfrak{D}_2\equiv0$.

Further observe that
 \begin{align}
\left(
\begin{array}{cc}
\frac{\partial F}{\partial\xi_1}-
\frac{\partial^2 F}{\partial\xi_1\partial\eta_1}\eta_1-
\frac{\partial^2 F}{\partial\xi_2\partial\eta_1}\eta_2-
\frac{\partial^2 F}{\partial\xi_3\partial\eta_1}\eta_3
\\
\frac{\partial F}{\partial\xi_2}-
\frac{\partial^2 F}{\partial\xi_1\partial\eta_2}\eta_1-
\frac{\partial^2 F}{\partial\xi_2\partial\eta_2}\eta_2-
\frac{\partial^2 F}{\partial\xi_3\partial\eta_2}\eta_3
\\
\frac{\partial F}{\partial\xi_3}-
\frac{\partial^2 F}{\partial\xi_1\partial\eta_3}\eta_1-
\frac{\partial^2 F}{\partial\xi_2\partial\eta_3}\eta_2-
\frac{\partial^2 F}{\partial\xi_3\partial\eta_3}\eta_3
\end{array}
\right)=-\frac{1}{2}|\eta|e^{-\frac{|\xi|^2}{4}}\left[ \xi -\frac{\langle\xi, \eta \rangle}{|\eta|^2}\eta\right],
\end{align}
and thus
\begin{align} \mathfrak{D}_1=
\left(
\begin{array}{cc} \partial_\tau v_1\\\partial_\tau v_2\\\partial_\tau v_3
\end{array}
\right)
 \left[
\left(
\begin{array}{cc}
\frac{\partial F}{\partial\xi_1}-
\frac{\partial^2 F}{\partial\xi_1\partial\eta_1}\eta_1-
\frac{\partial^2 F}{\partial\xi_2\partial\eta_1}\eta_2-
\frac{\partial^2 F}{\partial\xi_3\partial\eta_1}\eta_3
\\
\frac{\partial F}{\partial\xi_2}-
\frac{\partial^2 F}{\partial\xi_1\partial\eta_2}\eta_1-
\frac{\partial^2 F}{\partial\xi_2\partial\eta_2}\eta_2-
\frac{\partial^2 F}{\partial\xi_3\partial\eta_2}\eta_3
\\
\frac{\partial F}{\partial\xi_3}-
\frac{\partial^2 F}{\partial\xi_1\partial\eta_3}\eta_1-
\frac{\partial^2 F}{\partial\xi_2\partial\eta_3}\eta_2-
\frac{\partial^2 F}{\partial\xi_3\partial\eta_3}\eta_3
\end{array}
\right)+
 \frac{1}{2}\rho(\xi,\eta)\xi\right] \\
= \rho(\xi,\eta)\left(\frac{1}{2}\xi+\kappa\nu\right)\frac{\xta}{2|\eta|^2}\eta=
\frac{\xta^2}{4|\eta|^2}\rho(\xi,\eta)=\frac{1}{4}\big|\textup{Proj}_\eta \xi\big|^2\rho(\xi,\eta).\nonumber
\end{align}
We have obtained 
\begin{equation}
\frac{d}{d\tau}\int_{S^1} F(v_1,v_2,v_3,v_1',v_2',v_3')dx=
-\int_{S^1}\left(|\partial_\tau {\tilde \gamma}|^2-\frac{1}{4}\big|\textup{Proj}_\eta  {\tilde \gamma} \big|^2\right)\rho(v_1,v_2,v_3,v_1',v_2',v_3')dx.
\end{equation}
On the other hand
\begin{multline}\label{plusminus}
\left|\frac{1}{2}\xi+\kappa\nu\right|^2-\frac{1}{4}\big|\textup{Proj}_\eta \xi\big|^2=\\
\frac{1}{4}\left(\big|\textup{Proj}_\nu \xi\big|^2+
\big|\textup{Proj}_\eta \xi\big|^2+
\big|\textup{Proj}_{\nu\times\eta} \xi\big|^2\right)
+\kappa\langle\xi,\nu\rangle+\kappa^2-\frac{1}{4}\big|\textup{Proj}_\eta \xi\big|^2=\\
\frac{1}{4}\big|\textup{Proj}_{\nu\times\eta} \xi\big|^2+\left| \kappa+\frac{1}{2}\langle\xi,\nu\rangle\right|^2,
\end{multline}
implying Huisken's monotonicity formula (\ref{monformHu}) for the rescaled curve shortening flow in $\R^3$.

\begin{rem}
It has been shown in \cite{M} that using the stabilization technique one can not only verify but actually also re-discover
Huisken's monotonicity formula in $\R^2$. 
\end{rem}







\section{Computations of $\mathfrak{D}$ for a special class of functions $F$ and $\rho$} \label{sec:computD}

In the case of the Huisken's formula functions $F$ and
$\rho$ depend only on absolute values of $\xi$ and $\eta$. Generalizing the approach developed 
in \cite{M} for plane curves we are looking for formulas, which depend on 
the angle between $\xi$ and $\eta$. 

Taking into account the matrices (\ref{Frho}) we look for the functions $F$ and $\rho$
of the particular form
$$
F(\xi,\eta)=a(|\xi|)|\eta| f(\psi)
$$ 
and 
$$
\rho(\xi,\eta)=a(|\xi|)|\eta| g(\psi),
$$
with 
$$
\psi=\arcsin \frac{\langle\xi,\eta\rangle}{|\xi| |\eta|}
$$ 
being the angle between the position vector 
$\xi=\tilde{\gamma}$ and the plane orthogonal to the tangent $\eta=\tilde{\gamma}'$.
The functions $f(\psi)$, $g(\psi)$ and $a(r)$ will be specified in the upcoming sections.






\subsection{Computing $\mathfrak{D}_2$}\label{sec:d2}


\begin{lemma}\label{lem:D2Frho}
If
$$
f''+f=g
$$
then
$$
D^2_\eta F\,  \eta=\rho|\eta|^{-1}D^2_\eta|\eta|\, \eta=\textbf{0},
$$
$$
(D^2_\eta F-\rho|\eta|^{-1}D^2_\eta|\eta|) \, \xi=\textbf{0},
$$
and 
$$
(D^2_\eta F-\rho|\eta|^{-1}D^2_\eta|\eta|) (\xi\times\eta)=a(|\xi|)\frac{m(\psi)-g(\psi)}{|\eta|}(\xi\times\eta),
$$
where 
$$
m(\psi)=f(\psi)-f'(\psi)\tan\psi.
$$
\end{lemma}
\begin{proof}
Let us take
$$
A(\xi,\eta)=|\xi|^2|\eta|^2-\langle\xi,\eta\rangle^2=|\xi|^2|\eta|^2\cos^2\psi,
$$ 
then 
$$
D_\xi A=2|\eta|^2\xi-2\langle\xi,\eta\rangle\eta,\text{ and } 
D_\eta A=2|\xi|^2\eta-2\langle\xi,\eta\rangle\xi,
$$
$$
\langle \eta, D_\xi A \rangle=\langle \xi,D_\eta A\rangle=0,\text{ and } 
\langle \xi, D_\xi A \rangle=\langle \eta,D_\eta A\rangle=2A.
$$
Observe that
$$
\partial_{\eta_i}\psi=\frac{|\eta|^2\xi_i-\langle \xi,\eta \rangle \eta_i }{|\eta|^2A^{\frac{1}{2}}},
\,\,\,
\partial_{\xi_i}\psi=\frac{|\xi|^2\eta_i-\langle \xi,\eta \rangle \xi_i }{|\xi|^2A^{\frac{1}{2}}}.
$$

\begin{multline}
\partial^2_{\eta_i \eta_j}\psi=\frac{(2\xi_i\eta_j  -  \langle \xi,\eta \rangle \delta_{ij}-\xi_j\eta_i   )|\eta|^2 A^{\frac{1}{2}}
}{|\eta|^4 A}\\
-\frac{(|\eta|^2\xi_i-\langle \xi,\eta \rangle \eta_i )(2\eta_j   A^{\frac{1}{2}}   +|\eta|^2 A^{-\frac{1}{2}}   (|\xi|^2\eta_j-\xta\xi_j) )
}{|\eta|^4 A}\\
=|\eta|^{-4}A^{-1}\Big[
- |\eta|^2 A^{\frac{1}{2}} \langle \xi,\eta \rangle \delta_{ij} + \frac{|\eta|^4\xta}{A^{\frac{1}{2}}}\xi_i\xi_j  \\
-(\xi_i\eta_j+\xi_j\eta_i)\frac{|\xi|^2|\eta|^4}{A^{\frac{1}{2}}}+
 \frac{2A+|\xi|^2|\eta|^2}{A^{\frac{1}{2}}  }   \xta  \eta_i\eta_j 
\Big].
\end{multline}
Further
$$
\partial_{\eta_i}\left(|\eta|f(\psi) \right)=\partial_{\eta_i}|\eta| f(\psi)+|\eta|f'(\psi)\partial_{\eta_i}\psi 
$$
and
\begin{multline}
\partial^2_{\eta_i \eta_j}\left(|\eta| f(\psi) \right)
= f(\psi)\partial^2_{\eta_i \eta_j}|\eta|\\
+f'(\psi)\Big[ \partial_{\eta_j}|\eta|  \partial_{\eta_i}\psi + \partial_{\eta_i}|\eta|  \partial_{\eta_j}\psi   +
|\eta|\partial^2_{\eta_i\eta_j}\psi     \Big] \\
+f''(\psi) |\eta|\partial_{\eta_i}\psi  \partial_{\eta_j}\psi  =
\end{multline}
\begin{multline}
f(\psi)\partial^2_{\eta_i \eta_j}|\eta| \\
+f'(\psi)\Big[ -\frac{\xta}{|\eta|A^{\frac{1}{2}} } \delta_{ij}+|\eta|\frac{\xta}{A^{\frac{3}{2}}}\xi_i\xi_j 
-\frac{\xta^2}{|\eta| A^{\frac{3}{2}}}(\xi_i\eta_j+\xi_j\eta_i)+\frac{|\xi|^2}{|\eta|}\frac{\xta}{A^{\frac{3}{2}}}\eta_i\eta_j
\Big] \\
+f''(\psi)|\eta|\frac{|\eta|^4\xi_i\xi_j-|\eta|^2\langle \xi,\eta \rangle( \xi_i\eta_j+\xi_j\eta_i )+\xta^2\eta_i\eta_j }{|\eta|^4 A}.
\end{multline}
Thus for $\eta$
\begin{multline}
D^2_{\eta}\left(|\eta|f(\psi) \right)\,\eta=
f(\psi)D^2_{\eta}|\eta| \,\eta \\
+f'(\psi)\Big[ -\frac{\xta}{|\eta|A^{\frac{1}{2}} } \eta+|\eta|\frac{\xta^2}{A^{\frac{3}{2}}}\xi
-\frac{\xta^2}{|\eta| A^{\frac{3}{2}}}(|\eta|^2\xi+ \xta\eta)+\frac{|\xi|^2}{|\eta|}\frac{\xta}{A^{\frac{3}{2}}}|\eta|^2\eta
\Big] \\
+f''(\psi)|\eta|\frac{|\eta|^4\xta\xi -|\eta|^2\langle \xi,\eta \rangle(|\eta|^2\xi+ \xta\eta )+\xta^2|\eta|^2\eta  }{|\eta|^4 A}
=\textbf{0}+\textbf{0}+\textbf{0}
\end{multline}
since 
$$
D^2|\eta|\,\eta=|\eta|^{-3}
\left(
\begin{array}{ccc}
(\eta_2^2+\eta_3^2) & -\eta_1\eta_2 & -\eta_1\eta_3
\\
-\eta_1\eta_2 & (\eta_1^2+\eta_3^2) & -\eta_2\eta_3
\\
-\eta_1\eta_3 & -\eta_2\eta_3 & (\eta_1^2+\eta_2^2)
\end{array}
\right)\eta=\textbf{0}.
$$
On the other hand for $\xi$
\begin{multline}
D^2_{\eta}\left(|\eta|f(\psi) \right)\,\xi=
f(\psi)D^2_{\eta}|\eta| \,\xi \\
+f'(\psi)\Big[ -\frac{\xta}{|\eta|A^{\frac{1}{2}} } \xi+|\eta||\xi|^2\frac{\xta}{A^{\frac{3}{2}}}\xi
-\frac{\xta^2}{|\eta| A^{\frac{3}{2}}}(\xta\xi+ |\xi|^2\eta)+\frac{|\xi|^2}{|\eta|}\frac{\xta^2}{A^{\frac{3}{2}}}\eta
\Big] \\
+f''(\psi)|\eta|\frac{|\eta|^4|\xi|^2\xi -|\eta|^2\langle \xi,\eta \rangle(\xta\xi+ |\xi|^2\eta)+\xta^3\eta  }{|\eta|^4 A}=
\end{multline}
\begin{multline}
f(\psi)\frac{1}{|\eta|}\left(\xi-\frac{\xta}{|\eta|^2}\eta\right) +\textbf{0}+
f''(\psi)\frac{1}{|\eta|}\left(\xi-\frac{\xta}{|\eta|^2}\eta\right)= \\
\frac{1}{|\eta|}(f(\psi)+f''(\psi))\left(\xi-\frac{\xta}{|\eta|^2}\eta\right)=|\eta|^{-1}g(\psi)\left(\xi-\frac{\xta}{|\eta|^2}\eta\right),
\end{multline}
while
\begin{multline}
D^2|\eta|\,\xi=|\eta|^{-3}
\left(
\begin{array}{ccc}
(\eta_2^2+\eta_3^2) & -\eta_1\eta_2 & -\eta_1\eta_3
\\
-\eta_1\eta_2 & (\eta_1^2+\eta_3^2) & -\eta_2\eta_3
\\
-\eta_1\eta_3 & -\eta_2\eta_3 & (\eta_1^2+\eta_2^2)
\end{array}
\right)\xi=\\
\frac{1}{|\eta|^3}\left[|\eta|^2I-\big(\eta_i\eta_j\big)_{i,j}\right]\xi=
\frac{1}{|\eta|}\left(\xi-\frac{\xta}{|\eta|^2}\eta\right).
\end{multline}

\vspace{2mm}

\noindent Now let us consider the vector $\xi\times \eta$. First observe that
$$
D^2|\eta|\,(\xi\times\eta)=|\eta|^{-1}(\xi\times\eta).
$$
Then
\begin{multline}
D^2_{\eta}\left(|\eta|f(\psi) \right)\,(\xi\times\eta)=\left(f(\psi)
\frac{1}{|\eta|}-f'(\psi) \frac{\xta}{|\eta|A^{\frac{1}{2}} }\right)(\xi\times\eta)=\\
\frac{1}{|\eta|}\left(f(\psi)-\tan\psi f'(\psi)\right)(\xi\times\eta)=
\frac{m(\psi)}{|\eta|}(\xi\times\eta)
\end{multline}
and
$$
(D^2_\eta F-\rho|\eta|^{-1}D^2_\eta|\eta|) (\xi\times\eta)=a(|\xi|)|\eta|^{-1}\left(m(\psi)-
g(\psi)\right)(\xi\times\eta).
$$
\end{proof}

Lemma \ref{lem:D2Frho} shows that the matrices in (\ref{Frho}) do not coincide 
but it allows one to compute the following difference with $\mu=(v_1'',v_2'',v_3'')^T$:
$$
\left(D^2_\eta F(\xi, \eta)-\rho(\xi,\eta)|\eta|^{-1}D^2|\eta|\right)\mu=\\
\frac{a(|\xi|)}{|\eta|}\left(m(\psi)-g(\psi)
\right) \textup{Proj}_{\xi\times\eta}\mu.
$$
Substituting 
\begin{equation}  
\left(
\begin{array}{cc} \partial_\tau v_1\\\partial_\tau v_2\\\partial_\tau v_3
\end{array}
\right)=
\frac{1}{2}{\tilde \gamma}+
\frac{{\tilde \gamma}'\times ({\tilde \gamma}''\times{\tilde \gamma}')}{|{\tilde \gamma}'|^4}=
\frac{1}{2}\xi+
\frac{\eta\times (\mu\times\eta)}{|\eta|^4},
\end{equation}
we arrive at
\begin{multline}
\partial_\tau {\tilde\gamma} \left(D^2_\eta F(\xi, \eta)-\rho(\xi,\eta)|\eta|^{-1}D^2|\eta|\right){\tilde\gamma}''=\\
\frac{a(|\xi|)}{|\eta|}\left(m(\psi)-g(\psi)
\right)
\left(
\frac{1}{2}\xi+
\frac{\eta\times (\mu\times\eta)}{|\eta|^4}\right)\cdot \textup{Proj}_{\xi\times\eta}\mu.
\end{multline}
Observe that
$$
\eta\times (\mu\times\eta)=|\eta|^2\mu-\langle\mu,\eta\rangle\eta,
$$
and thus 
\begin{multline}
\left(
\frac{1}{2}\xi+
\frac{\eta\times (\mu\times\eta)}{|\eta|^4}\right)\cdot \textup{Proj}_{\xi\times\eta}\mu=
|\eta|^{-2}\mu \cdot\frac{(\xi\times\eta)\mu}{|(\xi\times\eta)|^2}(\xi\times\eta)=\\
 \frac{\text{Vol}(\xi,\eta,\mu)^2}{|\xi|^2|\eta|^{4}\cos^2\psi}= 
 \frac{(\xi\cdot(\eta\times\mu))^2}{|\xi|^2|\eta|^{4}\cos^2\psi}=
 \frac{\kappa^2|\eta|^{2}}{|\xi|^2\cos^2\psi}\left|\textup{Proj}_{\eta\times\nu}\xi\right|^2,
\end{multline}
where $\text{Vol}(\xi,\eta,\mu)=|(\xi\times\eta)\cdot\mu|=|\xi\cdot(\eta\times\mu)|$ is 
the volume of the parallelepiped formed by 
vectors $\xi, \eta,\mu$, and $\kappa=\frac{|\eta\times\mu|}{|\eta|^3}$ is the curvature.
In the last equality we have used that 
$$
\textup{Proj}_{\eta\times\mu}\xi=\textup{Proj}_{\eta\times\nu}\xi
.
$$
We have proven the following lemma.
\begin{lemma} 
Let $\mathfrak{D}_2$ be the expression defined in (\ref{defD2}), and $f$, $g$ and $m$ be as in Lemma~\ref{lem:D2Frho}. Then
\begin{multline}\label{deltaerku}
\mathfrak{D}_2=\partial_\tau {\tilde\gamma} \left(\rho(\xi,\eta)|\eta|^{-1}D^2|\eta|-D^2_\eta F(\xi, \eta)\right){\tilde\gamma}''=\\
-a(|{\tilde\gamma} |)|{\tilde\gamma}' |\left(m(\psi)-g(\psi)
\right)
 \frac{\kappa^2}{|{\tilde\gamma} |^2\cos^2\psi}\left|\textup{Proj}_{{\tilde\gamma}'\times\nu }{\tilde\gamma}  \right|^2.
 \end{multline}
\end{lemma}







\subsection{Computing $\mathfrak{D}_1$}

Now let us try to compute the difference of the terms which do not contain the second derivatives of ${\tilde\gamma}$:
\begin{align*}
\mathfrak{D}_1=
\frac{\rho(\xi,\eta)}{2}
\left(
\begin{array}{cc} \partial_\tau v_1\\\partial_\tau v_2\\\partial_\tau v_3
\end{array}
\right)\cdot{\tilde \gamma}\,\,\,
\\
+\partial_\tau v_1\left[\frac{\partial F}{\partial\xi_1}-
\frac{\partial^2 F}{\partial\xi_1\partial\eta_1}v_1'-
\frac{\partial^2 F}{\partial\xi_2\partial\eta_1}v_2'-
\frac{\partial^2 F}{\partial\xi_3\partial\eta_1}v_3'
\right]\,\,\,
\\
+\partial_\tau v_2\left[\frac{\partial F}{\partial\xi_2}-
\frac{\partial^2 F}{\partial\xi_1\partial\eta_2}v_1'-
\frac{\partial^2 F}{\partial\xi_2\partial\eta_2}v_2'-
\frac{\partial^2 F}{\partial\xi_3\partial\eta_2}v_3'
\right]\,\,\,
\\
+\partial_\tau v_3\left[\frac{\partial F}{\partial\xi_3}-
\frac{\partial^2 F}{\partial\xi_1\partial\eta_3}v_1'-
\frac{\partial^2 F}{\partial\xi_2\partial\eta_3}v_2'-
\frac{\partial^2 F}{\partial\xi_3\partial\eta_3}v_3'
\right].
\end{align*}
Let us compute 
$$
\partial_{\xi_i}F-\eta D_{\xi}\left(\frac{\partial F}{\partial\eta_i}\right)=
\frac{\partial F}{\partial\xi_i}-
\frac{\partial^2 F}{\partial\xi_1\partial\eta_i}\eta_1-
\frac{\partial^2 F}{\partial\xi_2\partial\eta_i}\eta_2-
\frac{\partial^2 F}{\partial\xi_3\partial\eta_i}\eta_3
$$
for
$$
F(\xi,\eta)= a(|\xi|)|\eta|f(\psi).
$$
Substituting
$$
\partial_{\xi_i}\psi=\frac{1}{\cos\psi}\frac{\eta_i}{|\xi||\eta|}-\tan\psi\frac{\xi_i}{|\xi|^2}\text{ and }
\partial_{\eta_i}\psi=\frac{1}{\cos\psi}\frac{\xi_i}{|\xi||\eta|}-\tan\psi\frac{\eta_i}{|\eta|^2},
$$
we obtain
\begin{multline}
\frac{\partial F}{\partial\xi_i}=\frac{|\eta|}{|\xi|} a'(|\xi|)f(\psi)\xi_i+|\eta|a(|\xi|)f(\psi)\partial_{\xi_i}\psi=\\
\frac{|\eta|}{|\xi|}\left( a'(|\xi|)f(\psi)-\frac{a(|\xi|)}{|\xi|}f'(\psi)\tan\psi
\right)\xi_i+
\frac{a(|\xi|)}{|\xi|}\frac{f'(\psi)}{\cos\psi}\eta_i,
\end{multline}
and
\begin{multline}
\frac{\partial^2 F}{\partial\xi_i\partial\eta_j}=
\frac{1}{|\xi||\eta|}\left( a'(|\xi|)f(\psi)-\frac{a(|\xi|)}{|\xi|}f'(\psi)\tan\psi
\right)\xi_i\eta_j\\
+
\frac{|\eta|}{|\xi|}\left( a'(|\xi|)f'(\psi)-\frac{a(|\xi|)}{|\xi|}(f'(\psi)\tan\psi)'
\right)\xi_i\partial_{\eta_j}\psi\\
+
\frac{a(|\xi|)}{|\xi|}\frac{f'(\psi)}{\cos\psi}\delta_{ij}+
\frac{a(|\xi|)}{|\xi|}\left(\frac{f'(\psi)}{\cos\psi}\right)'\eta_i\partial_{\eta_j}\psi,
\end{multline}
and
\begin{multline}
\eta D_{\xi}\left(\frac{\partial F}{\partial\eta_j}\right)=
\frac{1}{|\xi||\eta|}\left( a'(|\xi|)f(\psi)-\frac{a(|\xi|)}{|\xi|}f'(\psi)\tan\psi
\right)\xta\eta_j\\
+\frac{|\eta|}{|\xi|}\left( a'(|\xi|)f'(\psi)-\frac{a(|\xi|)}{|\xi|}(f'(\psi)\tan\psi)'
\right)\xta\partial_{\eta_j}\psi\\
+\frac{a(|\xi|)}{|\xi|}\frac{f'(\psi)}{\cos\psi}\eta_j+
\frac{a(|\xi|)}{|\xi|}\left(\frac{f'(\psi)}{\cos\psi}\right)'|\eta|^2\partial_{\eta_j}\psi=
\\
\left( a'(|\xi|)f(\psi)\sin\psi-\frac{a(|\xi|)}{|\xi|}f'(\psi)\tan\psi\sin\psi+
\frac{a(|\xi|)}{|\xi|}\frac{f'(\psi)}{\cos\psi}
\right)\eta_j\\
+
|\eta|^2\left( a'(|\xi|)f'(\psi)\sin\psi-\frac{a(|\xi|)}{|\xi|}(f'(\psi)\tan\psi)'\sin\psi+
\frac{a(|\xi|)}{|\xi|}\left(\frac{f'(\psi)}{\cos\psi}\right)'
\right)\partial_{\eta_j}\psi.
\end{multline}
Further
\begin{multline}
\partial_{\xi_i}F-\eta D_{\xi}\left(\frac{\partial F}{\partial\eta_i}\right)=\\
\frac{|\eta|}{|\xi|}\left( a'(|\xi|)f(\psi)-\frac{a(|\xi|)}{|\xi|}f'(\psi)\tan\psi
\right)\xi_i+
\frac{a(|\xi|)}{|\xi|}\frac{f'(\psi)}{\cos\psi}\eta_i\\
-
\left( a'(|\xi|)f(\psi)\sin\psi-\frac{a(|\xi|)}{|\xi|}f'(\psi)\tan\psi\sin\psi+
\frac{a(|\xi|)}{|\xi|}\frac{f'(\psi)}{\cos\psi}
\right)\eta_j\\
-|\eta|^2\left( a'(|\xi|)f'(\psi)\sin\psi-\frac{a(|\xi|)}{|\xi|}(f'(\psi)\tan\psi)'\sin\psi+
\frac{a(|\xi|)}{|\xi|}\left(\frac{f'(\psi)}{\cos\psi}\right)'
\right)\partial_{\eta_j}\psi=
\end{multline}
\begin{multline}
=\frac{|\eta|}{|\xi|}\Bigg( a'(|\xi|)f(\psi)-\frac{a(|\xi|)}{|\xi|}f'(\psi)\tan\psi-
a'(|\xi|)f'(\psi)\tan\psi\\
+\frac{a(|\xi|)}{|\xi|}(f'(\psi)\tan\psi)'\tan\psi-
\frac{a(|\xi|)}{|\xi|}\frac{1}{\cos\psi}\left(\frac{f'(\psi)}{\cos\psi}\right)'
\Bigg)\xi_i\\
-\Bigg( a'(|\xi|)f(\psi)\sin\psi-\frac{a(|\xi|)}{|\xi|}f'(\psi)\tan\psi\sin\psi+
\frac{a(|\xi|)}{|\xi|}\frac{f'(\psi)}{\cos\psi}-
\frac{a(|\xi|)}{|\xi|}\frac{f'(\psi)}{\cos\psi}\\
-\tan\psi\left[
a'(|\xi|)f'(\psi)\sin\psi-\frac{a(|\xi|)}{|\xi|}(f'(\psi)\tan\psi)'\sin\psi+
\frac{a(|\xi|)}{|\xi|}\left(\frac{f'(\psi)}{\cos\psi}\right)'
\right]
\Bigg)\eta_i,
\end{multline}
and observing that
$$
\left(\frac{f'(\psi)}{\cos\psi}\right)'-(f'(\psi)\tan\psi)'\sin\psi= f''(\psi)\cos\psi,
$$
we arrive at
\begin{multline}\label{final:comput}
\partial_{\xi_i}F-\eta D_{\xi}\left(\frac{\partial F}{\partial\eta_i}\right)=\\
\left[
a'(|\xi|)f(\psi)-\left(a'(|\xi|)+\frac{a(|\xi|)}{|\xi|}\right)\tan\psi f'(\psi)-\frac{a(|\xi|)}{|\xi|}f''(\psi)
\right]
\left(
\frac{|\eta|}{|\xi|}\xi_i-\sin\psi\eta_i
\right)=\\
\left[
\left(a'(|\xi|)+\frac{a(|\xi|)}{|\xi|}\right)(f(\psi)-\tan\psi f'(\psi))-\frac{a(|\xi|)}{|\xi|}g(\psi)
\right]
\left(
\frac{|\eta|}{|\xi|}\xi_i-\sin\psi\eta_i
\right)=\\
\left[
\left(a'(|\xi|)+\frac{a(|\xi|)}{|\xi|}\right)m(\psi)-\frac{a(|\xi|)}{|\xi|}g(\psi)
\right]
\left(
\frac{|\eta|}{|\xi|}\xi_i-\sin\psi\eta_i
\right)=\\
a(|\xi|)|\eta|
\left[\frac{1}{|\xi|}
\left(\frac{a'(|\xi|)}{a(\xi)}+\frac{1}{|\xi|}\right)m(\psi)-\frac{1}{|\xi|^2}g(\psi)
\right]
\left(
\xi_i-\frac{|\xi|}{|\eta|}\sin\psi\eta_i
\right).
\end{multline}

We have proven the following lemma.
\begin{lemma}
Let $\mathfrak{D}_1$ be the expression defined in (\ref{eqFrho}), and $f$, $g$ and $m$ be as in Lemma~\ref{lem:D2Frho}. Then
\begin{multline}\label{deltamek}
\mathfrak{D}_1=
a(|\xi|)|\eta|
\left[\frac{1}{|\xi|}
\left(\frac{a'(|\xi|)}{a(|\xi|)}+\frac{1}{|\xi|}\right)m(\psi)+\left(\frac{1}{2}-\frac{1}{|\xi|^2}\right)g(\psi)
\right]\xi\cdot \partial_\tau{\tilde\gamma}\\
-
a(|\xi|)|\xi|
\left[\frac{1}{|\xi|}
\left(\frac{a'(|\xi|)}{a(|\xi|)}+\frac{1}{|\xi|}\right)m(\psi)-\frac{1}{|\xi|^2}g(\psi)
\right]
\sin\psi \, \eta\cdot\partial_\tau{\tilde\gamma},
\end{multline}
where $\xi=\tilde{\gamma}$ and $\eta=\tilde{\gamma}'$.
\end{lemma}
\vspace{5mm}

\section{New monotonicity formulas}
\label{sec:newmon}


\subsection{Proof of the Theorem \ref{thm2}}

\begin{proof}
Let us first observe that (\ref{cosformul2}) follows from (\ref{deltamek}) and (\ref{deltaerku})
with $f(\psi)=\sin\psi$ resulting $$
\frac{d}{d\tau}
\int_{S^1}
a(|{\tilde\gamma}|)|{\tilde\gamma}'|\sin\psi
dx=0
$$
implies that the integral in (\ref{cosformul2}) must be a constant. This constant is zero because of the 
known results about convergence to Abresch-Langer curves or a Grim Reaper and 
their symmetry. 

But the equation (\ref{cosformul2}) is rather simple (almost trivial) and can be proven directly. Observe that the expression
$$
\int_{x_1}^{x_2}
|{\tilde\gamma}'|\sin\psi
dx
$$
measures the change of the distance of the point ${\tilde\gamma}(\tau, x)$ from the origin, as
the parameter $x$ varies from $x_1$ to $x_2$. This makes the proof trivial for an arbitrary
closed curve ${\tilde\gamma}$ and a step-function $a(r)$.
The proof follows now by approximation. 

\vspace{3mm}
 
The proof of (\ref{cosformul}) satisfying the condition  (\ref{noturnpsi}) follows from (\ref{deltamek}) and (\ref{deltaerku}) with 
$$
f(\psi)=\cos\psi,\,\,\,a(r)=r^{-1}.
$$ 

For the simplicity let us first consider the case of the plane curves and prove the equation (\ref{cos2D}) from the Corollary \ref{thm3}.

Since the formula (\ref{cosformul}) in $\R^3$ is correct with respect to any reference point we will write it 
for the plane curve 
with respect to the point 
$(0,0,\epsilon)$ and pass to limit $\epsilon\searrow 0+$ (see Figure~3). Obviously the condition (\ref{noturnpsi}) 
is satisfied if we take $(0,0,\epsilon)$ as the reference  point. 
We have
\begin{multline}
\frac{d}{d\tau}\int_{S^1}
\frac{|{\tilde\gamma}_\epsilon'|}{|{\tilde\gamma}_\epsilon|}\cos\psi
dx=
\int_{S^1}
\frac{|{\tilde\gamma}_\epsilon'|}{|{\tilde\gamma}_\epsilon|^3\cos^3\psi}\kappa^2
|\textup{Proj}_{\nu\times{\tilde\gamma}_\epsilon'} {\tilde\gamma}_\epsilon|^2 
dx=\\
\int_{\cos \psi<\delta}
\frac{|{\tilde\gamma}_\epsilon'|}{|{\tilde\gamma}_\epsilon|^3\cos^3\psi}\kappa^2
|\textup{Proj}_{\nu\times{\tilde\gamma}_\epsilon'} {\tilde\gamma}_\epsilon|^2 
dx+
\int_{\cos \psi\geq \delta}
\frac{|{\tilde\gamma}_\epsilon'|}{|{\tilde\gamma}_\epsilon|^3\cos^3\psi}\kappa^2
|\textup{Proj}_{\nu\times{\tilde\gamma}_\epsilon'} {\tilde\gamma}_\epsilon|^2 
dx\\
=I_1+I_2.
\end{multline}
\begin{figure}[h!]
\begin{center}
\includegraphics[scale=0.36]{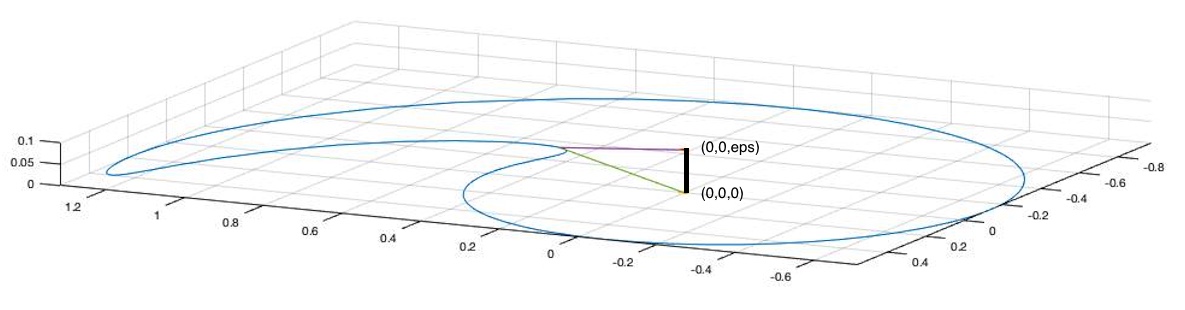}\label{thm32d3d}
\caption{}
\end{center}
\end{figure}
First let us observe that for arbitrary fixed $\delta>0$
$$
|I_2|\leq \delta^{-3} \int_{S_1} \frac{\kappa^2}{|{\tilde \gamma}_\epsilon|}
\left|\textup{Proj}_{\nu\times{\tilde\gamma}_\epsilon'} \frac{{\tilde\gamma}_\epsilon}{|{\tilde\gamma}_\epsilon|}\right|^2 
|{\tilde\gamma}_\epsilon'|dx\to 0.
$$
Let us now choose $\delta>0$ small enough, such that 
$$
\{x\in S_1\, |\, \cos\psi<\delta\}=\cup_{j\in J} (x_j-\omega_j,x_j+\omega_j),
$$ 
where $\cos\psi(x_j)=0$ and the intervals $ (x_j-\omega_j,x_j+\omega_j)$ are disjoint.
Let us pick one of these intervals, which we without loss of generality assume to be 
$(-\omega,\omega)$. 
We will compute the limit $\epsilon\to 0$ of the following integral
\begin{multline}
\int_{-\omega}^\omega 
\frac{|{\tilde\gamma}_\epsilon'|}{|{\tilde\gamma}_\epsilon|^3\cos^3\psi}\kappa^2
|\textup{Proj}_{\nu\times{\tilde\gamma}_\epsilon'} {\tilde\gamma}_\epsilon|^2 
dx
=\\
\int_{-\omega}^\omega \frac{|{\tilde\gamma}'_\epsilon|^4}{(|{\tilde\gamma}_\epsilon|^2|{\tilde\gamma}'_\epsilon|^2-\langle {\tilde\gamma}_\epsilon,{\tilde\gamma}'_\epsilon\rangle^2
)^\frac{3}{2}}
\left(\frac{|{\tilde\gamma}'_\epsilon\times{\tilde\gamma}''_\epsilon|}{|{\tilde\gamma}'_\epsilon|^3}\right)^2
|\textup{Proj}_{\nu\times{\tilde\gamma}_\epsilon'} {\tilde\gamma}_\epsilon|^2 
dx=\\
\int_{-\omega}^\omega \frac{1}{(|{\tilde\gamma}_\epsilon|^2|{\tilde\gamma}'_\epsilon|^2-\langle {\tilde\gamma}_\epsilon,{\tilde\gamma}'_\epsilon\rangle^2
)^\frac{3}{2}}
\frac{|{\tilde\gamma}'_\epsilon\times{\tilde\gamma}''_\epsilon|^2}{|{\tilde\gamma}'_\epsilon|^2}
|\textup{Proj}_{\nu\times{\tilde\gamma}_\epsilon'} {\tilde\gamma}_\epsilon|^2 
dx.
\end{multline}
In order to compute this limit 
we approximate the curve (after rotation) in the interval $x\in (-\omega,\omega)$ by the parabola
\begin{equation}\label{parabola}
{\tilde\gamma}_\epsilon(x)=
\big(|{\tilde\gamma}(0)| +x ,
\tfrac{1}{2}\kappa(0)x^2,
-\epsilon\big)
+
(0,O(x^3),0),
\end{equation}
with
$$
{\tilde\gamma}'_\epsilon(x)=
(1 ,
\kappa(0)x,
0)
+
(0,O(x^2),0)
\,\,\,
\text{ and }
\,\,\,\,
{\tilde\gamma}''_\epsilon(x)=
(0 ,
\kappa(0),
0)
+
(0,O(x),0),
$$
and arrive at

\begin{multline}\label{replaceparab}
\int_{-\omega}^\omega \frac{1}{(|{\tilde\gamma}_\epsilon|^2|{\tilde\gamma}'_\epsilon|^2-\langle {\tilde\gamma}_\epsilon,{\tilde\gamma}'_\epsilon\rangle^2
)^\frac{3}{2}}
\frac{|{\tilde\gamma}'_\epsilon\times{\tilde\gamma}''_\epsilon|^2}{|{\tilde\gamma}'_\epsilon|^2}
|\textup{Proj}_{\nu\times{\tilde\gamma}_\epsilon'} {\tilde\gamma}_\epsilon|^2 
dx
\approx\\
\int_{-\omega}^\omega \frac{1}{(\epsilon^2(1+\kappa^2x^2)+\kappa^2x^2(\frac{x}{2}+|{\tilde\gamma}|)^2)^\frac{3}{2}}
\frac{\kappa^2}{(1+\kappa^2x^2)}
\epsilon^2
dx=\\
\frac{\kappa}{|{\tilde\gamma}|}\int_{-\omega}^\omega
 \frac{1}{\left((1+\kappa^2x^2)+\frac{\kappa^2|{\tilde\gamma}|^2x^2}{\epsilon^2}(\frac{x}{2|{\tilde\gamma}|}+1)^2\right)^\frac{3}{2}}
 \frac{1}{(1+\kappa^2x^2)}
d \underbrace{\frac{\kappa |{\tilde\gamma}|x}{\epsilon}}_{=\sigma}=\\
\frac{\kappa}{|{\tilde\gamma}|}\int_{-\tfrac{\kappa |{\tilde\gamma}|\omega}{\epsilon}}^{\tfrac{\kappa |{\tilde\gamma}|\omega}{\epsilon}}
 \frac{1}{(1+\frac{\epsilon^2\sigma^2}{|{\tilde\gamma}|^2})\left((1+\frac{\epsilon^2\sigma^2}{|{\tilde\gamma}|^2})+
\sigma^2(\frac{\epsilon\sigma}{2|{\tilde\gamma}|^2\kappa}+1)^2\right)^\frac{3}{2}}
d\sigma\to_{\epsilon\to 0}
\\
\frac{\kappa}{|{\tilde\gamma}|}
\int_{-\infty}^\infty 
\frac{1}{(1+\sigma^2)^\frac{3}{2}}
d\sigma=2
\frac{\kappa}{|{\tilde\gamma}|},
\end{multline}
where starting line two we write $\kappa$ for $\kappa(0)$ and $|{\tilde\gamma}|$ for $|{\tilde\gamma}(0)|$, as well as use
$$
\int_{-\infty}^\infty 
\frac{1}{(1+\sigma^2)^\frac{3}{2}}
d\sigma=\int_{-\infty}^\infty 
\frac{1}{(\cosh t)^2}
d t=2.
$$
What remains to observe is that replacing the curve by the parabola in the second line of (\ref{replaceparab})
was justified:
\begin{multline}\label{replace2}
\bigg|\int_{-\omega}^\omega \frac{1}{(|{\tilde\gamma}_\epsilon|^2|{\tilde\gamma}'_\epsilon|^2-\langle {\tilde\gamma}_\epsilon,{\tilde\gamma}'_\epsilon\rangle^2
)^\frac{3}{2}}
\frac{|{\tilde\gamma}'_\epsilon\times{\tilde\gamma}''_\epsilon|^2}{|{\tilde\gamma}'_\epsilon|^2}
|\textup{Proj}_{\nu\times{\tilde\gamma}_\epsilon'} {\tilde\gamma}_\epsilon|^2 
dx
-\\
\int_{-\omega}^\omega \frac{1}{(\epsilon^2(1+\kappa^2x^2)+\kappa^2x^2(\frac{x}{2}+|{\tilde\gamma}|)^2)^\frac{3}{2}}
\frac{\kappa^2}{(1+\kappa^2x^2)}
\epsilon^2
dx\bigg|=\\
\bigg|
\int_{-\omega}^\omega \frac{1}{(\epsilon^2(1+\kappa^2x^2)+\kappa^2x^2(\frac{x}{2}+|{\tilde\gamma}|)^2+O(x^3))^\frac{3}{2}}
\frac{(\kappa+O(x))^2}{(1+\kappa^2x^2+O(x^3))}
\epsilon^2
dx
-\\
\int_{-\omega}^\omega \frac{1}{(\epsilon^2(1+\kappa^2x^2)+\kappa^2x^2(\frac{x}{2}+|{\tilde\gamma}|)^2)^\frac{3}{2}}
\frac{\kappa^2}{(1+\kappa^2x^2)}
\epsilon^2
dx\bigg|=\\
\bigg|\int_{-\omega}^\omega \frac{1}{(\epsilon^2(1+\kappa^2x^2)+\kappa^2x^2(\frac{x}{2}+|{\tilde\gamma}|)^2)^\frac{3}{2}}
\frac{O(x)}{(1+\kappa^2x^2)}
\epsilon^2
dx\bigg|\leq\\
M\epsilon\int_{-\infty}^\infty 
\frac{|\sigma|}{(1+\sigma^2)^\frac{3}{2}}
d\sigma\to_{\epsilon\to 0}0,
\end{multline}
where the last inequality follows from the computations in (\ref{replaceparab}), with $M$ 
being a large enough constant depending on $\tilde\gamma$.

\vspace{3mm}

The proof the equation (\ref{cosformul}) in $\R^3$ follows by the same argument. One needs only to 
observe that 
if we 
replace the parabola (\ref{parabola}) by the non-plane curve
$$
{\tilde\gamma}_\epsilon(x)=
\big(|{\tilde\gamma}(0)| +x ,
\tfrac{1}{2}\kappa(0)x^2,
-\epsilon \big)
+
(0,O(x^3),O(x^3)),
$$
then in the computations (\ref{replace2}) instead of
$|\textup{Proj}_{\nu\times{\tilde\gamma}_\epsilon'} {\tilde\gamma}_\epsilon|^2 =\epsilon^2$
we will have 
$$|\textup{Proj}_{\nu\times{\tilde\gamma}_\epsilon'} {\tilde\gamma}_\epsilon|^2=\epsilon^2(1+O(x))^2,$$
which would not change the vanishing limit.
\end{proof}




\subsection{Proof of the Theorem \ref{thm1}}
\begin{proof}
We will not only verify the statement of the theorem, but rather present how the formula (\ref{monformHG}) is being derived.

The second term in (\ref{deltamek}) is a ``good'' one since
$$
\eta\cdot\partial_\tau{\tilde\gamma}=\eta\cdot \left(\frac{1}{2} \xi+\kappa\nu  \right)=\frac{1}{2}\xta=\frac{1}{2}|\xi||\eta|\sin\psi.
$$
Our strategy now is to make the first term in (\ref{deltamek}) to vanish, which is only possible if 
$$
m(\psi)=\lambda g(\psi).
$$
Ideally we would be happy to have $\lambda=1$, which would make $\mathfrak{D}_2=0$, but as we will see, this
will lead to $f(\psi)=g(\psi)=\text{const}$ and $a(r)=e^{-\frac{r^2}{4}}$, i.e., Huisken's formula. Indeed, we have
\begin{equation}\label{fcond1}
f+f''=g
\end{equation}
and we want in addition 
\begin{equation}\label{fcond2}
m(\psi)=f(\psi)-f'(\psi)\tan\psi=\lambda g(\psi). 
\end{equation}
Differentiating the latter equation and using $f''=g-f$ we obtain
\begin{multline*}
f'(\psi)-f'(\psi)\frac{1}{\cos^2\psi}-f''(\psi)\tan\psi=\\
-f'(\psi)\tan^2\psi+f(\psi)\tan\psi-g(\psi)\tan\psi=\lambda g'(\psi).
\end{multline*}
Substituting $m(\psi)=f(\psi)-f'(\psi)\tan\psi=\lambda g(\psi)$ we arrive at
$$
(\lambda-1)g(\psi)\tan\psi=\lambda g'(\psi).
$$
Solving
$$
\frac{g'(\psi)}{g(\psi)}=\frac{\lambda-1}{\lambda}\tan\psi
$$
we obtain
$$
g(\psi)=\left(\frac{1}{\cos\psi}\right)^{\frac{\lambda-1}{\lambda}}.
$$
This is of course only a necessary condition, and we need to find an appropriate $f_\lambda$, satisfying (\ref{fcond1}) and (\ref{fcond2}). The general solution to (\ref{fcond1}) is
\begin{equation}\label{flambgen}
\int_0^\psi g_\lambda(t)\sin(\psi-t)dt+c_1\cos\psi+c_2\sin\psi.
\end{equation}
Computing
\begin{multline}
\int_0^\psi g_\lambda(t)\sin(\psi-t)dt=\\
\sin\psi\int_0^\psi (\cos t)^\frac{1}{\lambda}dt+\cos\psi( \lambda(\cos \psi)^{\frac{1}{\lambda}}-\lambda)=\\
\sin\psi\int_0^\psi (\cos t)^\frac{1}{\lambda}dt+\lambda (\cos \psi)^{1+\frac{1}{\lambda}}-\lambda\cos\psi,
\end{multline}
hence we chose for $f_\lambda$ in (\ref{flambgen}) $c_2=0$ and $c_1=\lambda$, which leads to (\ref{fcond2}).

To make the first term in (\ref{deltamek}) vanish we now solve the equation for $a(r)$
$$
\lambda
\left(\frac{a'(r)}{a(r)}+\frac{1}{r}\right)+\frac{r}{2}-\frac{1}{r}=0,
$$
and obtain
$$
a(r)=e^{-\frac{r^2}{4\lambda}}r^{\frac{1-\lambda}{\lambda}}.
$$
As a result using (\ref{main}) we obtain
\begin{equation}\label{deltamekhuis}
\mathfrak{D}_1=\frac{1}{2}
a(|\xi|)|\xi|
g(\psi)
\sin\psi \, \eta\cdot\partial_\tau{\tilde\gamma}
=\frac{1}{4}
a(|\xi|)|\xi|^2|\eta|
g(\psi)
\sin^2\psi=
\frac{1}{4}
\rho(\xi, \eta)
\left|\textup{Proj}_{\eta} \xi \right|^2,
\end{equation}
and
\begin{multline}
\mathfrak{D}_2=-a(|\xi|)
(\lambda-1)g(\psi)
 \frac{|\eta|\kappa^2}{|\xi|^2\cos^2\psi}\left|\textup{Proj}_{\eta\times\nu}\xi\right|^2=\\
- (\lambda-1)\rho(\xi, \eta)
 \frac{\kappa^2}{|\xi|^2\cos^2\psi}\left|\textup{Proj}_{\eta\times\nu}\xi\right|^2.
 \end{multline}
Due to (\ref{plusminus}) we have
\begin{multline}
\frac{d}{d\tau}\int_{S^1}F_\lambda({\tilde\gamma},{\tilde\gamma}')dx=\\
-\int_{S^1}
\left(
\frac{1}{4}|\textup{Proj}_{\nu\times{\tilde\gamma}'} {\tilde\gamma}|^2
+\left| \kappa+\frac{1}{2}\langle{\tilde\gamma},\nu\rangle\right|^2
\right)
\rho_\lambda({\tilde\gamma},{\tilde\gamma}')dx\\
-
(\lambda-1)\int_{S^1}
|\textup{Proj}_{\nu\times{\tilde\gamma}'} {\tilde\gamma}|^2 \frac{\kappa^2}{|{\tilde\gamma}|^2\cos^2\psi}
\rho_\lambda({\tilde\gamma},{\tilde\gamma}')dx=\\
-\int_{S^1}
\left| \kappa+\frac{1}{2}\langle{\tilde\gamma},\nu\rangle\right|^2
\rho_\lambda({\tilde\gamma},{\tilde\gamma}')dx\\
-
\int_{S^1}
\left(\frac{1}{4}+(\lambda-1)\frac{\kappa^2}{|{\tilde\gamma}|^2\cos^2\psi}\right)
|\textup{Proj}_{\nu\times{\tilde\gamma}'} {\tilde\gamma}|^2 
\rho_\lambda({\tilde\gamma},{\tilde\gamma}')dx.
\end{multline}
\end{proof}



\subsection{Proof of the Theorem \ref{thm4}}
\begin{proof}
Similarly to the previous proof we need to compute (\ref{deltamek}) 
and (\ref{deltaerku}) for the particular choice of the function
$F$.
To simplify the computations let us write
$$
F(\xi,\eta)=a_1(|\xi|)|\eta|f_1(\psi)-a_2(|\xi|)|\eta|f_2(\psi),
$$
where 
$$
a_1(r)=r^{-1}\, ,\,\,\,f_1(\psi)=h(\psi)\, ,\,\,\,a_2(r)=b(r)\, ,\,\,\,f_2(\psi)=\cos\psi.
$$
By design 
$$
h''(\psi)+h(\psi)=\frac{1}{\cos\psi}
$$
and thus
$$
m_1(\psi)=\frac{\log\cos\psi}{\cos\psi}\, ,\,\,\,g_1(\psi)=m_2(\psi)=\frac{1}{\cos\psi}\, ,\,\,\,
g_2(\psi)=0.
$$
Moreover, since 
$$
\frac{a_1'(r)}{a_1(r)}+\frac{1}{r}=0,
$$  
and
$$
a_2'(r)+\frac{1}{r}a_2(r)=\frac{1}{2}-\frac{1}{r^2},
$$  
we can easily substitute functions above into  (\ref{deltamek}) and 
compute $\mathfrak{D}_1$ with $\xi=\tilde\gamma$ and $\eta={\tilde\gamma}'$:
\begin{multline}
\mathfrak{D}_1=
\frac{|\eta|}{|\xi|}\left(\frac{1}{2}-\frac{1}{|\xi|^2}\right)\frac{1}{\cos\psi}
\xi\cdot \partial_\tau{\tilde\gamma}
+
\frac{1}{|\xi|^2}\frac{1}{\cos\psi}
\sin\psi \, \eta\cdot\partial_\tau{\tilde\gamma}\\
-\frac{|\eta|}{|\xi|}
\left(\frac{1}{2}-\frac{1}{|\xi|^2}\right)\frac{1}{\cos\psi}
\xi\cdot \partial_\tau{\tilde\gamma}
+
\left(\frac{1}{2}-\frac{1}{|\xi|^2}\right)\frac{1}{\cos\psi}
\sin\psi \, \eta\cdot\partial_\tau{\tilde\gamma}\\
=\frac{1}{2}\frac{1}{\cos\psi}
\sin\psi \, \eta\cdot\partial_\tau{\tilde\gamma}=
\frac{1}{4}\frac{|\eta|}{|\xi|}\frac{1}{\cos\psi}\left| \textup{Proj}_\eta \xi\right|^2,
\end{multline}
where in the last step we use (\ref{main}), like in (\ref{deltamekhuis}).
Similarly, following  (\ref{deltaerku}) we obtain
\begin{multline}
\mathfrak{D}_2=\\
-\Big[a_1(|{\tilde\gamma} |)\left(m_1(\psi)-g_1(\psi)\right)-a_2(|{\tilde\gamma} |)\left(m_2(\psi)-g_2(\psi)
\right)\Big]
 \frac{|{\tilde\gamma}' |\kappa^2}{|{\tilde\gamma} |^2\cos^2\psi}\left|\textup{Proj}_{{\tilde\gamma}'\times\nu }{\tilde\gamma}  \right|^2\\
 =\Big(1+|\xi|b(|\xi|)-\log\cos\psi\Big)
 \frac{|{\tilde\gamma}' |\kappa^2}{|{\tilde\gamma} |^3\cos^3\psi}\left|\textup{Proj}_{{\tilde\gamma}'\times\nu }{\tilde\gamma}  \right|^2.
  \end{multline}
  This together with (\ref{plusminus}) completes the proof.
\end{proof}


\vspace{3mm}

{\bf Declarations.} 
The author has no conflicts of interest to declare that are relevant to the content of this article.
There are no data associated with this research.



\end{document}